\documentclass[11pt]{amsart}

\usepackage{amsmath}              
\usepackage{amsfonts}              
\usepackage{amsthm} 
\usepackage[pdftex]{graphicx}
\usepackage{mdwlist}
\usepackage{amssymb}
\usepackage{multicol}
\usepackage{pifont}
\usepackage{wrapfig}

\newtheorem{thm}{Theorem}[section]
\newtheorem{lem}[thm]{Lemma}
\newtheorem{prop}[thm]{Proposition}
\newtheorem{cor}[thm]{Corollary}
\newtheorem{conj}[thm]{Conjecture}

\newtheorem{procedure}[thm]{Procedure}

\theoremstyle{definition}
\newtheorem{example}{Example}[section]

\theoremstyle{definition}
\newtheorem{defn}[thm]{Definition}

\theoremstyle{plain}

\allowdisplaybreaks[3]  

\begin{document}

\title[The SGIS of Points in General Position]{The Asymptotic Behaviour of Symbolic Generic Initial Systems of Points in General Position}
\author{Sarah Mayes}

\maketitle  
\vspace*{-2em} 

\begin{abstract}
Consider the ideal $I \subseteq K[x,y,z]$ corresponding to points $p_1, \dots, p_r$ of $\mathbb{P}^2$.  We study the \textit{symbolic generic initial system} $\{ \text{gin}(I^{(m)}) \}_m$ of such an ideal and its behaviour as $m$ gets large.  In particular, we describe the \textit{limiting shape} of this system explicitly when $p_1, \dots, p_r$ lie in general position using the SHGH Conjecture for $r \geq 9$.  The symbolic generic initial system and its limiting shape reflects information about the Hilbert functions of fat point ideals.
\end{abstract}

\section{Introduction}

Generic initial ideals can be viewed as a coordinate-independent version of initial ideals, which carry much of the same information as the initial ideal with the added benefit of preserving, and even revealing, certain geometric information.  Given an ideal $I \subseteq K[x,y,z]$ of distinct points $p_1, \dots, p_r$ in $\mathbb{P}^2$, the reverse lexicographic generic initial ideal of $I$, $\text{gin}(I)$, can detect if a subset of the points lies on a curve of a certain degree (see \cite{EP90} or Theorem 4.4 of \cite{Green98}).  If we instead consider the ideal $I^{(m)}$ of the fat point subscheme $Z_m = mp_1+\cdots +mp_r \subseteq \mathbb{P}^2$, one might ask what $\text{gin}(I^{(m)})$ says about $Z_m$; this question motivated the work in this paper.



Despite being simple to describe, ideals $I^{(m)}$ of fat point subschemes $Z_m=m(p_1+\cdots+p_r)$ have proven difficult to understand.  For example, there are still many open problems and unresolved conjectures related to finding the Hilbert function of $I^{(m)}$ and even the degree $\alpha(I^{(m)})$ of the smallest degree element of $I^{(m)}$.  Many of the challenges in understanding the individual ideals $I^{(m)}$ can be overcome by changing one's focus to studying the general behaviour of the entire family of ideals $\{I^{(m)}\}_m$. For instance, more can be said about the Seshadri constant 
$$\epsilon(I) = \lim_{m \rightarrow \infty} \frac{\alpha(I^{(m)})}{rm}$$ 
than the invariants $\alpha(I^{(m)})$ of each ideal (see \cite{BC10} and \cite{Harbourne02} for further background on these constants).  Thus, we will explore the asymptotic behaviour of the entire symbolic generic initial system $\{ \text{gin}(I^{(m)})\}_m$ as a first step to understanding the generic initial ideals of fat point subschemes.  

To describe limiting behaviour, we define the \textit{limiting shape} $P$ of the symbolic generic initial system $\{ \text{gin}(I^{(m)}\}$ of the ideal $I \subseteq K[x,y,z]$ corresponding to an arrangement of points in $\mathbb{P}^2$ to be the limit $\lim_{m \rightarrow \infty} \frac{1}{m} P_{\text{gin}(I^{(m)})}$, where $P_{\text{gin}(I^{(m)})}$ denotes the Newton polytope of $\text{gin}(I^{(m)})$.  We will see that each of the ideals $\text{gin}(I^{(m)})$ is generated in the variables $x$ and $y$, so that $P_{\text{gin}(I^{(m)})}$, and thus $P$, can be thought of as a subset of $\mathbb{R}^2$.  One reason for studying the limiting shape of a system of monomial ideals is that it completely determines the asymptotic multiplier ideals of the system (see \cite{Howald01} and \cite{Mayes12a}).

When the point arrangement has an ideal $I$ that is a complete intersection of type $(\alpha, \beta)$ with $\alpha \leq \beta$, a special case of the main result of \cite{Mayes12a} shows that the limiting shape of the symbolic generic initial system has a boundary defined by the line through the points $(\alpha, 0)$ and $(0, \beta)$.  The main result of this paper is the following theorem describing the limiting shape of the symbolic generic initial system of an ideal of $r$ distinct points of $\mathbb{P}^2$ in general position, assuming that the SHGH Conjecture \ref{conj:shgh} holds for the case where $r \geq 9$.


\begin{thm}
\label{thm:mainthm}
Let $I \subseteq R=K[x,y,z]$ be the ideal of $r>1$ distinct points $p_1, \dots, p_r$ of $\mathbb{P}^2$ in general position and $P$ be the limiting shape of the reverse lexicographic symbolic generic initial system $\{ \text{gin}(I^{(m)})\}_m$.  Then $P$ can be characterized as follows.
\begin{enumerate}
\item[(a)]  If $r \geq 9$ and the SHGH Conjecture holds for infinitely many $m$, then $P$ has a boundary defined by the line through the points $(\sqrt{r},0)$ and $(0, \sqrt{r})$.  See Figure \ref{fig:largenumberpolytope}.
\item[(b)]  If $6 \leq r<9$, then $P$ has a boundary defined by the line through the points $(\gamma_1,0)$ and $(0, \gamma_2)$ where:
\begin{enumerate}
\item[(i)]  $\gamma_1 = \frac{12}{5}$ and $\gamma_2 = \frac{5}{2}$ when $r=6$;
\item[(ii)] $\gamma_1 = \frac{21}{8}$ and $\gamma_2 = \frac{8}{3}$ when $r=7$; and
\item[(iii)] $\gamma_1 = \frac{48}{17}$ and $\gamma_2 = \frac{17}{6}$ when $r=8$.
\end{enumerate}
\item[(c)] If $r=4$ or $r=5$, then $P$ has a boundary defined by the line through the points $(2,0)$ and $(0, \frac{r}{2})$.  If $r=2$ or $r=3$, then $P$ has a boundary defined by the line through the points $(\frac{r}{2},0)$ and $(0,2)$.
\end{enumerate}
\end{thm}

\begin{figure}
\begin{center}
\includegraphics[width=4cm]{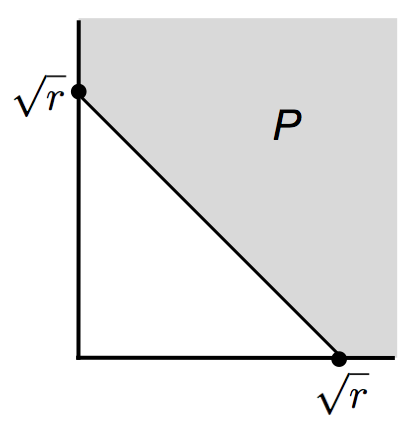}
\end{center}
\caption{The limiting shape $P$ of $\{\text{gin}(I^{(m)})\}_m$ where $I$ is the ideal of $r \geq 9$ points in general position, assuming that the SHGH Conjecture holds for infinitely many $m$.}
\label{fig:largenumberpolytope}
\end{figure}

Precisely what information is carried by the limiting shape of the symbolic generic initial system of other point arrangements is still uncertain.  While one can prove that the $x$-intercept of the boundary of $P$ is equal to $r \epsilon(I)$ (see Section \ref{sec:prelim}), that the $y$-intercept reflects the asymptotic behaviour of the regularity of the ideals $I^{(m)}$ (see \cite{Mayes12a}), and that the volume under $P$ is equal to $\frac{r}{2}$ (Proposition \ref{prop:volumesequal}), there is likely additional geometric information encoded within $P$.  Two important questions concern the form of $P$: is $P$ always a polytope, and what does it mean for the boundary of $P$ to be defined by a certain number of line segments?

Following background information in Section \ref{sec:prelim}, the three parts of Theorem \ref{thm:mainthm} are proven in Sections \ref{sec:largenumber}, \ref{sec:678}, and \ref{sec:smallnumber}.  The final section contains an example demonstrating that there are point arrangements for which the boundary of the limiting polytope of the symbolic generic initial system is not defined by a single line segment.

\section*{Acknowledgements}

I thank Karen Smith for introducing this problem to me and for many useful discussions.  I also thank Brian Harbourne and Susan Cooper for helping me to learn about fat points.
\section{Preliminaries}
\label{sec:prelim}

In this section we will introduce some notation, definitions, and preliminary results related to fat points in $\mathbb{P}^2$, generic initial ideals, and systems of ideals. Unless stated otherwise, $R=K[x,y,z]$ is the polynomial ring in three variables over a field $K$ of characteristic 0 with the standard grading and some fixed term order $>$ with $x>y>z$.


\subsection{Fat Points in $\mathbb{P}^2$}
\label{sec:fatpoints}

\begin{defn}
Let $p_1, \dots, p_r$ be distinct points of $\mathbb{P}^2$, $I_j$ be the ideal of $K[\mathbb{P}^2] = R$ consisting of all forms vanishing at the point $p_j$, and $I = I_1 \cap \cdots \cap I_r$ be the ideal of the points $p_1, \dots, p_r$.  A \textbf{fat point subscheme} $Z = m_1p_1 + \cdots + m_rp_r$, where the $m_i$ are nonnegative integers, is the subscheme of $\mathbb{P}^2$ defined by the ideal $I_Z = I_1^{m_1} \cap \cdots \cap I_r^{m_r}$ consisting of forms that vanish at the points $p_i$ to multiplicity at least $m_i$.  When $m_i = m$ for all $i$, we say that $Z$ is \textbf{uniform}; in this case, $I_Z$ is equal to the $m^{\text{th}}$ symbolic power of $I$, $I^{(m)}$.
\end{defn}



The following lemma relates the symbolic and ordinary powers of $I$ in the case we are interested in (see, for example, Lemma 1.3 of \cite{AV03}).

\begin{lem} 
\label{lem:symbpower}
If $I$ is the ideal of distinct points in $\mathbb{P}^2$,
$$(I^m)^{\text{sat}} = I^{(m)},$$
where $J^{\text{sat}} = \bigcup_{k \geq 0}(J:\mathfrak{m}^k)$ denotes the saturation of $J$.
\end{lem}

In this paper we will be interested in studying the ideals of uniform fat point subschemes $Z = mp_1 + \cdots + mp_r$ such that the points $p_1, \dots, p_r$ are in \textit{general position}.

\begin{defn}
A collection of points in $\mathbb{P}^2$ is in \textbf{general position} if, for each $d \in \mathbb{N}$, no subset of cardinality ${d+2 \choose 2}$ lies on any curve of degree $d$.
\end{defn}

\subsection{Generic Initial Ideals}

An element $g = (g_{ij}) \in \text{GL}_n(K)$ acts on $R = K[x_1, \dots, x_n]$ and sends any homogeneous element $f(x_1, \dots, x_n)$ to the homogeneous element 
$$f(g(x_1), \dots, g(x_n))$$ 
where $g(x_i) = \sum_{j=1}^n g_{ij}x_j$.  If $g(I)=I$ for every upper triangular matrix $g$ then we say that $I$ is \textit{Borel-fixed}.  Borel-fixed ideals are \textit{strongly stable} when $K$ is of characteristic 0; that is, for every monomial $m$ in the ideal such that $x_i$ divides $m$, the monomials $\frac{x_jm}{x_i}$ are also in the ideal for all $j<i$.  This property makes such ideals particularly nice to work with.

To any homogeneous ideal $I$ of $R$ we can associate a Borel-fixed monomial ideal $\text{gin}_{>}(I)$ which can be thought of as a coordinate-independent version of the initial ideal.  Its existence is guaranteed by Galligo's theorem (also see \cite[Theorem 1.27]{Green98}).

\begin{thm}[{\cite{Galligo74} and \cite{BS87b}}]
\label{thm:galligo}
For any multiplicative monomial order $>$ on $R$ and any homogeneous ideal $I\subset R$, there exists a Zariski open subset $U \subset \text{GL}_n$ such that $\text{In}_{>}(g(I))$ is constant and Borel-fixed for all $g \in U$.  
\end{thm}

\begin{defn}
The \textbf{generic initial ideal of $I$}, denoted $\text{gin}_{>}(I)$, is defined to be $\text{In}_{>}(g(I))$ where $g \in U$ is as in Galligo's theorem.
\end{defn}

The \textit{reverse lexicographic order} $>$ is a total ordering on the monomials of $R$ defined by: 
\begin{enumerate}
\item if $|I| =|J|$ then $x^I > x^J$ if there is a $k$ such that $i_m = j_m$ for all $m>k$ and $i_k < j_k$; and
\item if $|I| > |J|$ then $x^I >x^J$.
\end{enumerate}
For example, $x_1^2 >x_1x_2 > x_2^2>x_1x_3>x_2x_3>x_3^2$.   From this point on, $\text{gin}(I) = \text{gin}_{>}(I)$ will denote the generic initial ideal with respect to the reverse lexicographic order.

Recall that the Hilbert function $H_I(t)$ of $I$ is defined by $H_I(t) = \text{dim}(I_t)$.  The following theorem is a consequence of the fact that Hilbert functions are invariant under making changes of coordinates and taking initial ideals; we will use it frequently and freely throughout this paper.

\begin{thm}
\label{thm:commonproperties}
For any homogeneous ideal $I$ in $R$, the Hilbert functions of $I$ and $\text{gin}(I)$ are equal.
\end{thm}

In this paper we will be studying the set of reverse lexicographic generic initial ideals of symbolic powers  of a fixed ideal $I$, $\{\text{gin}(I^{(m)})\}_m$.  One reason for our interest in these ideals is the following proposition which tells us that we can get information about the ideals  $\text{gin}(I^{m})$ from the ideals $ \text{gin}(I^{(m)})$.

\begin{prop}[{Proposition 2.21 of \cite{Green98}}]
\label{prop:saturations}
Fix the reverse lexicographic order on $K[x_1, \dots, x_n]$ with $x_1> x_2 > \cdots > x_n$ and let $\mathfrak{m} = (x_1, \dots, x_n)$.  Then, if $I^{\text{sat}} = \bigcup_{k \geq 0}(I:\mathfrak{m}^k)$ denotes the saturation of $I$,
$$\text{gin}(I^{\text{sat}}) = \bigcup_{k \geq 0}(\text{gin}(I):\mathfrak{m}^k) = (\text{gin}(I))^{\text{sat}}.$$
In particular, when $I$ is the ideal of distinct points in $\mathbb{P}^2$,
$$\text{gin}(I^{(m)}) =  \bigcup_{k \geq 0}(\text{gin}(I^m):\mathfrak{m}^k) = (\text{gin}(I^m))^{\text{sat}}.$$
for all $m \geq 1$ by Lemma \ref{lem:symbpower}.
\end{prop}

The following result due to Bayer and Stillman (\cite{BS87}).

\begin{prop}[{Theorem 2.21 of \cite{Green98}}]
\label{prop:geninxy}
Fix the reverse lexicographic order on $K[x_1, \dots, x_n]$ with $x_1> x_2 > \cdots > x_n$. An ideal $I$ of $R$ is saturated if and only if no minimal generator of $\text{gin}(I)$ involves the variable $x_n$.  In particular, when $I \subset K[x,y,z]$ is the (saturated) ideal of a set of distinct points of $\mathbb{P}^2$, no minimal generator of $\text{gin}(I^{(m)})$ involves the variable $z$.
\end{prop}

\begin{cor}
\label{cor:formofgin}
 Suppose that $I \subset K[x,y,z]$ is the ideal of a set of distinct points of $\mathbb{P}^2$.  Then the minimal generators of $\text{gin}(I^{(m)})$ under the reverse lexicographic order are of the form
$$\{ x^{\alpha(m)}, x^{\alpha(m)-1}y^{\lambda_{\alpha(m)-1}(m)}, \dots, xy^{\lambda_1(m)}, y^{\lambda_0(m)} \}$$
where $\lambda_0(m) > \lambda_1(m) > \cdots > \lambda_{\alpha(m)-1}(m) \geq 1$.  
\end{cor}

\begin{proof}
By a result of Herzog and Srinivasan relating the dimension of a Borel-fixed monomial ideal $J$ to the variable powers that it contains, $\text{gin}(I^{(m)})$ contains a power of $y$ (see Lemma 3.1 of \cite{HS98}).   Now the result is immediate from Proposition \ref{prop:geninxy} and the fact that $\text{gin}(I^{(m)})$ is a Borel-fixed ideal.
\end{proof}


\subsection{Graded Systems}

In this subsection we introduce some tools for studying certain collections of monomial ideals.

\begin{defn}[\cite{ELS01}]
A \textbf{graded system of ideals} is a collection of ideals $J_{\bullet} =\{J_i\}_{i=1}^{\infty}$ such that 
$$J_i \cdot J_j \subseteq J_{i+j} \hspace{0.3in} \text{ for all } i, j \geq 1.$$
\end{defn}

\begin{defn}
The \textbf{generic initial system} of a homogeneous ideal $I$ is the collection of ideals $J_{\bullet}$ such that $J_i = \text{gin}(I^i)$.  The \textbf{symbolic generic initial system} of a homogeneous ideal $I$ is the collection of ideals $J_{\bullet}$ such that $J_i = \text{gin}(I^{(i)})$.
\end{defn}

\begin{lem}
The symbolic generic initial system is a graded system of ideals.
\end{lem}

\begin{proof}
By definition, $\text{gin}(I^{(i)})$ is a monomial ideal. We need to show that for all $i,j \geq 1$, $\textnormal{gin}(I^{(i)})\cdot \textnormal{gin}(I^{(j)}) \subseteq \textnormal{gin}(I^{(i+j)})$.  For any $l \geq 1$, let $U_l$ be the Zariski open subset of $GL_n$ such that $\text{gin}(I^{(l)})= \text{In}(g \cdot (I^{(l)}))$ for all $g$ in $U_l$.  Since $U_i$, $U_j$, and $U_{i+j}$ are Zariski open they have a nonempty intersection; fix some $g \in U_i \cap U_j \cap U_{i+j}$.  
Given monomials $f' \in \textnormal{gin}(I^{(i)}) = \textnormal{In}(g(I^{(i)}))$ and $h' \in \textnormal{gin}(I^{(j)}) = \textnormal{In}(g(I^{(j)}))$, suppose that $f'=\textnormal{In}(g(f))$ and $h' = \textnormal{In}(g(h))$ for $f \in I^{(i)}$ and $h \in I^{(j)}$.  Now 
 $$f'\cdot h' = \textnormal{In}(g(f)) \textnormal{In}(g(h)) = \textnormal{In}(g(f) \cdot g(h)) = \textnormal{In}(g(f \cdot h)) \in \text{In}(g(I^{(i+j)}))$$
since $f \cdot h \in I^{(i+j)}$.\footnote{This holds since the set of symbolic powers of a fixed ideal is itself a graded system: $I^{(i)} \cdot I^{(j)} \subseteq I^{(i+j)}$.}  Thus $f'\cdot h' \in \text{gin}(I^{(i+j)})$ as desired.
\end{proof}

The same proof with $I^{(i)}$ replaced by $I^i$ shows that the generic initial system is also a graded system of ideals.  


\begin{defn} [\cite{ELS03}]
\label{defn:algebraicvolume}
Let $\mathrm{a}_{\bullet}$ be a graded system of zero-dimensional ideals in $R=K[x_1, \dots, x_n]$.  The \textbf{volume} of $\mathrm{a}_{\bullet}$ is 
$$\mathrm{vol}(\mathrm{a}_{\bullet}) := \limsup_{m \rightarrow \infty} \frac{n! \cdot \mathrm{length}(R/\mathrm{a}_m)}{m^n}.$$
\end{defn}

Let $J$ be a monomial ideal of $R$.  We may associate to $J$ a subset $\Lambda$ of $\mathbb{N}^n$ consisting of the points $\lambda$ such that $x^{\lambda} \in J$.  The \textit{Newton polytope} $P_J$ of $J$ is the convex hull of $\Lambda$ regarded as a subset of $\mathbb{R}^n$.  Scaling the polytope $P_J$ by a factor of $r$ gives another polytope which we will denote $rP_J$.

If $\mathrm{a}_{\bullet}$ is a graded system of monomial ideals in $R$, the polytopes of $\{ \frac{1}{q} P_{\mathrm{a}_q} \}_q$ are nested: $\frac{1}{c}P_{\mathrm{a}_c} \subset \frac{1}{c+1}P_{\mathrm{a}_{c+1}}$ for all $c \geq 1$.  The \textit{limiting shape $P$ of $\mathrm{a}_{\bullet}$} is the limit of the polytopes in this set:
$$P = \bigcup_{q \in \mathbb{N}^*} \frac{1}{q} P_{\mathrm{a}_q}.$$

Under the additional assumption that the ideals of $\mathrm{a}_{\bullet}$ are zero-dimensional, the closure of each set $\mathbb{R}^n_{\geq 0} \backslash P_{\mathrm{a}_q}$ in $\mathbb{R}^n$ is compact. This closure is denoted by $Q_q$ and we let 
$$Q = \bigcap_{q \in \mathbb{N}^*} \frac{1}{q} Q_q.$$

\begin{prop} [\cite{Mustata02}]
\label{prop:volumesequal}
If $\textrm{a}_{\bullet}$ is a graded system of zero-dimensional monomial ideals in $R=K[x_1, \dots, x_n]$ and $Q$ is as defined above,
$$\mathrm{vol}(\textrm{a}_{\bullet}) = n! \mathrm{ vol}(Q).$$
\end{prop}

\begin{proof}
This is an immediate consequence of Theorem 1.7 and Lemma 2.13 of \cite{Mustata02}.
\end{proof}

We now turn our attention to using the concept of the limiting shape to study the asymptotic behaviour of the system of ideals $\{ \text{gin}(I^{(m)})\}_m$ where $I$ is an ideal of $r$ distinct points in $\mathbb{P}^2$.  By Corollary \ref{cor:formofgin}, the ideals $\text{gin}(I^{(m)})$ for such an $I$ are generated in the variables $x$ and $y$ and contain a power of both $x$ and $y$.  Therefore, we can think of the ideals $\text{gin}(I^{(m)})$ as zero-dimensional in $K[x,y]$ and consider a two dimensional limiting shape $P$ of the symbolic generic initial system.

\begin{lem}
\label{lem:volume}
Suppose that $I$ is the ideal of $r$ distinct points $p_1, p_2, \dots, p_r$ in $\mathbb{P}^2$ and $J_m = \text{gin}(I^{(m)}) \subseteq K[x,y]$.  If $P$ is the limiting shape of $J_{\bullet}$ and $Q \subseteq \mathbb{R}^2$ is as above,
$$\text{vol}(Q) = \frac{r}{2}.$$
\end{lem}

\begin{proof}
Let $h = ax+by+cz$ be a general linear form in $K[x,y,z]$.  To reduce our calculations to $K[x,y]$, consider the ring isomorphism 
$$\phi: \frac{K[x,y,z]}{(h)} \rightarrow K[x,y]$$
 given by sending $x$ to $x$, $y$ to $y$, and $z$ to $-\frac{a}{c}x - \frac{b}{c}y$.  If $I_i\subseteq K[x,y,z]$ is the ideal of the point $p_i$ in $\mathbb{P}^2$ then $\phi(\overline{I_i}) \cong (x,y)^m$.  Further, $\phi(\overline{I^{(m)}}) = \phi(\overline{I_1^m} \cap \cdots \cap \overline{I_r^m})$ and $\text{length} \Big( \frac{K[x,y]}{(x,y)^m} \Big) = {m+1\choose 2}$ so

\begin{eqnarray*}
\text{length} \bigg( \frac{K[x,y]}{\phi(\overline{I^{(m)}})}\bigg) &=& \text{length} \bigg( \frac{K[x,y]}{I_1^m} \times \cdots \times \frac{K[x,y]}{I_r^m} \bigg)\\
&=& \text{length} \bigg(  \frac{K[x,y]}{(x,y)^m} \times \cdots \times \frac{K[x,y]}{(x,y)^m} \bigg)\\
&=&r(1+ \cdots +m).
\end{eqnarray*}

The fact that $\text{gin}(I^{(m)})$ is generated in $x$ and $y$ (Proposition \ref{prop:geninxy}) together with a well-known relation between the generic initial ideals of $J$ and $\phi(\overline{J})$ (see Corollary 2.5 of \cite{Green98}) imply that $\text{gin}(I^{(m)})$ and $\text{gin}(\phi(\overline{I^{(m)}}))$ have the same generators.  Thus, thinking of $\text{gin}(I^{(m)})$ as being contained in $K[x,y]$,

\begin{eqnarray*}
\text{length} \bigg( \frac{K[x,y]}{\text{gin}(I^{(m)})}\bigg) &=& \text{length} \bigg( \frac{K[x,y]}{\text{gin}(\phi(\overline{I^{(m)}})}\bigg)\\
&=& \text{length} \bigg( \frac{K[x,y]}{\phi(\overline{I^{(m)}})}\bigg)\\
&=& r(1+\cdots +m) = r\Big( \frac{m^2+m}{2} \Big).
\end{eqnarray*}

Therefore,
\begin{eqnarray*}
\text{vol}(Q) &=& \lim_{m \rightarrow \infty} \frac{\text{length}(K[x,y]/\text{gin}(I^{(m)}))}{m^2}\\
&=& \lim_{m \rightarrow \infty} \frac{(m^2+m)r}{2m^2} \\
&=& \frac{r}{2}.
\end{eqnarray*}
\end{proof}

If $I$ is the ideal of distinct points in $\mathbb{P}^2$, the minimal generating set of each ideal $\text{gin}(I^{(m)})$ contains a power of $x$ and a power of $y$, say $x^{\alpha(m)}$ and $y^{\zeta(m)}$ by Corollary \ref{cor:formofgin}.  It is clear that $\lim_{m\rightarrow \infty} \frac{\alpha(m)}{m}$ and $\lim_{m\rightarrow \infty} \frac{\zeta(m)}{m}$ are the $x$- and $y$-intercepts  of the limiting shape $P$ of $\{\text{gin}(I^{(m)})\}_m$.

\begin{cor}
\label{cor:InterceptsDetermineShape}
Let $I \subseteq K[x,y,z]$ be the ideal of $r$ distinct points in $\mathbb{P}^2$ and $P$ be the limiting shape of the symbolic generic initial system $\{  \text{gin}(I^{(m)}) \}_m$.  Suppose that the $x$-intercept $\gamma_1$ and the $y$-intercept $\gamma_2$ of the boundary of $P$ are such that $\gamma_1 \cdot \gamma_2 = r$.  Then the limiting polytope $P$ has a boundary defined by the line passing through $(\gamma_1, 0)$ and $(0, \gamma_2)$.
\end{cor}

\begin{proof}
The smallest possible limiting shape $P$ satisfying the given conditions is the one defined by the line segment through $(\gamma_1, 0)$ and $(0, \gamma_2)$ since $P$ is convex by definition.   This extreme case is the only one in which the maximum volume under $P$ is achieved, in which case $\text{vol}(Q) = \frac{\gamma_1 \gamma_2}{2}$.  Under the assumptions stated, $\gamma_1 \cdot \gamma_2 = r$ so, by the previous lemma, the maximum volume must be attained and $P$ is as claimed.
\end{proof}

\section{The Symbolic Generic Initial System of Greater than 8 Uniform Points in General Position}
\label{sec:largenumber}

Throughout this section, $I \subseteq R[x,y,z]$ will denote the ideal of $r \geq 9$ points $p_1, \dots, p_r$ of $\mathbb{P}^2$ in general position.  We will frequently use the fact that the Hilbert function of an ideal and its generic initial ideal are equal (see Theorem \ref{thm:commonproperties}).

Computing the Hilbert functions of ideals of fat points in $\mathbb{P}^2$ can be very difficult.  However, the following conjecture of Segre, Harbourne, Gimigliano, and Hirschowiz proposes that when $Z$ is the ideal of  $r \geq 9$ uniform fat points in general position, $H_{I_Z}(t)$ has a very simple form.  See \cite{Harbourne12} for a statement similar to what follows and \cite{Harbourne02} for more general versions of the conjecture.


\begin{conj}[SHGH Conjecture]
\label{conj:shgh}
Let $R=K[x,y,z]$ and $I$ be the ideal of $r \geq 9$ generic points $p_i \in \mathbb{P}^2$.  Then, if $I^{(m)}$ is the ideal of the uniform fat point subscheme $Z = m(p_1+\cdots+p_r)$,
$$H_{I^{(m)}}(t) = \max \bigg \{ {t+2 \choose 2} - r {m+1 \choose 2}, 0 \bigg \}.$$
\end{conj}

The SHGH Conjecture is known to hold for certain special cases.  For example, it holds for infinitely many $m$ when $r$ is a square by \cite{HR04}, and for all $m$ when $r$ is a square not divisible by a prime bigger than 5 by \cite{Evain99}.   


The main goal of this section is to prove the first part of Theorem \ref{thm:mainthm}.

\newtheorem*{thm:mainthma}{Theorem \ref{thm:mainthm}(a)}
\begin{thm:mainthma}
\label{thm:LargeAsymptoticBehaviour}
Fix $r \geq 9$ points of $\mathbb{P}^2$ in general position and suppose that the SHGH Conjecture \ref{conj:shgh} holds for infinitely many $m$.  Let $I$ be the ideal of $r$ general points in $\mathbb{P}^2$ and $P$ be the the limiting shape of the reverse lexicographic symbolic generic initial system $\{\text{gin}(I^{(m)})\}_m$.  Then the boundary of $P$ is defined by the line through the points $(\sqrt{r},0)$ and $(0, \sqrt{r})$.
\end{thm:mainthma}

The proof of this statement is contained in Section \ref{sec:proofofLargeNumber}.  In preparation for this proof, we compute the minimal generators of the generic initial ideals $\text{gin}(I^{(m)})$ in Section \ref{sec:StructureofLargeNumber} under the assumption that the SHGH Conjecture holds.

\subsection{Structure of $\text{gin}(I^{(m)})$}
\label{sec:StructureofLargeNumber}

The following lemma records the degree of the smallest degree element of $I^{(m)}$.

\begin{lem}
\label{lem:alphavalue}
Let $I$ be the ideal of $p_1, \dots, p_r$ points of $\mathbb{P}^2$ in general position where $r \geq 9$ and suppose that $\alpha(m)$ is the least integer $t$ such that $H_{I^{(m)}}(t) >0$. Then, if the SHGH Conjecture holds for $Z=m(p_1+\cdots p_r)$, 
$$\alpha(m)  =\bigg  \lfloor -\frac{1}{2} + \sqrt{\frac{1}{4} + rm^2+rm} \bigg \rfloor.$$
\end{lem}

\begin{proof}
By the SHGH Conjecture, $\alpha(m)$ is the smallest integer $t$ such that ${t+2 \choose 2} - r{m+1 \choose 2} > 0$.
\begin{eqnarray*}
&&\frac{(t+2)(t+1)}{2} - r\frac{(m+1)m}{2} >0\\
&\Leftrightarrow& t^2+3t+2-rm^2-rm > 0
\end{eqnarray*}
If this is an equality, the positive root is 
$$t = - \frac{3}{2} + \frac{1}{2} \sqrt{1+4rm^2+4rm}.$$
Then the least integer that will make the expression positive is 
$$\alpha(m)  = \bigg \lfloor -\frac{3}{2} + \frac{1}{2}\sqrt{1 + 4rm^2+4rm} +1 \bigg \rfloor.$$
\end{proof}

If the SHGH Conjecture holds, the structure of the generic initial ideals $\text{gin}(I^{(m)})$ is very simple.  
\begin{prop}
\label{prop:gensofgin}
Let $I$ be the ideal of $r \geq 9$ points of $\mathbb{P}^2$ in general position, fix a non-negative integer $m$, and suppose that the SHGH Conjecture holds for $I^{(m)}$.  Set $\alpha = \alpha(m)$ and $\eta := H_{I^{(m)}}(\alpha) = {\alpha+2 \choose 2} - r{m+1 \choose 2}$ so that $\eta \leq \alpha+1$.  Then 
$$\text{gin}(I^{(m)}) = (x^{\alpha}, x^{\alpha-1}y, \dots, x^{\alpha-\eta+1}y^{\eta-1}, x^{\alpha-\eta}y^{\eta+1}, x^{\alpha-\eta-1}y^{\eta+2}, \dots, xy^{\alpha}, y^{\alpha+1})$$
when $\eta < \alpha+1$ and
$$\text{gin}(I^{(m)}) = (x^{\alpha}, x^{\alpha-1}y, \dots, xy^{\alpha-1}, y^{\alpha})$$
when $\eta = \alpha+1$.
\end{prop}

\begin{proof} Since there is no element of $\text{gin}(I^{(m)})$ of degree smaller than $\alpha(m)$, all monomials of degree $\alpha(m)$ in $\text{gin}(I^{(m)})$ must be generators and thus contain only the variables $x$ and $y$ by Proposition \ref{prop:geninxy}. There are at most $\alpha+1$ monomials of degree $\alpha$ in the variables $x$ and $y$ so $\eta := H_{\text{gin}(I^{(m)})}(\alpha) \leq \alpha+1$.

If $\eta = \alpha+1$ then all $\alpha+1$ monomials of degree $\alpha$ in the variables $x$ and $y$ are minimal generators of $\text{gin}(I^{(m)})$.  By Corollary \ref{cor:formofgin}, $\text{gin}(I^{(m)})$ has exactly $\alpha+1$ minimal generators.  Thus, all minimal generators of $\text{gin}(I^{(m)})$ are of degree $\alpha$ and are the ones given.

Now suppose that $\eta < \alpha+1$.  The $\eta$ monomials of $\text{gin}(I^{(m)})$ of degree $\alpha$ must be minimal generators.  In fact, since generic initial ideals are Borel-fixed, these must be the largest $\eta$ monomials in $x$ and $y$ of degree $\alpha$ with respect to the reverse lexicographic order: 
$$ [\text{gin}(I^{(m)})]_{\alpha}= \{ x^{\alpha}, x^{\alpha-1}y, \dots, x^{\alpha-\eta+1}y^{\eta-1} \}.$$
There are exactly $\eta$ elements of $\text{gin}(I^{(m)})$ of degree $\alpha+1$ involving the variable $z$, obtained by multiplying each of the $\eta$ generators of $[\text{gin}(I^{(m)})]_{\alpha}$ by $z$.  By the SHGH Conjecture \ref{conj:shgh},
\begin{eqnarray*}
H_{I^{(m)}}(\alpha+1) - \eta &=& \bigg[ {\alpha+1+2 \choose 2} - r{m+1\choose 2} \bigg] - \bigg[ {\alpha+2 \choose 2} - r{m+1\choose 2} \bigg]\\
&=& {\alpha+2+1 \choose 2} - {\alpha+2 \choose 2} \\
&=& {\alpha+2 \choose 1} = \alpha+2
\end{eqnarray*}
and there are $\alpha+2$ monomials in $\text{gin}(I^{(m)})$ of degree $\alpha+1$ containing only the variables $x$ and $y$.  Since there are exactly $\alpha+2$ monomials of degree $\alpha+1$ in $x$ and $y$,  $\text{gin}(I^{(m)})$ contains all of them.  Thus, the remaining generators of $\text{gin}(I^{(m)})$ are of degree $\alpha+1$; they are
$$x^{\alpha-\eta}y^{\eta+1}, x^{\alpha-\eta-1}y^{\eta+2}, \dots, xy^{\alpha}, y^{\alpha+1}$$
by Corollary \ref{cor:formofgin}.
\end{proof}

\subsection{Proof of Theorem \ref{thm:mainthm} (a)}
\label{sec:proofofLargeNumber}

\begin{proof}[Proof of Theorem~\ref{thm:mainthm} (a)]

By Proposition \ref{prop:gensofgin}, $x^{\alpha(m)}$ and $y^{\alpha(m)+1}$ or $y^{\alpha(m)}$ are the smallest variable powers contained in $\text{gin}(I^{(m)})$ for all $m$ such that the SHGH Conjecture holds.  Thus,  the $x$-intercept of the boundary of $P$ is 
$$\lim_{m \rightarrow \infty} \frac{\alpha(m)}{m}$$
while the $y$-intercept of the boundary of $P$ is 
$$\lim_{m \rightarrow \infty} \frac{\alpha(m)+1}{m} = \lim_{m \rightarrow \infty} \frac{\alpha(m)}{m}$$
where we take the limits over the infinite subset such that the SHGH Conjecture holds.

By Lemma \ref{lem:alphavalue}, 
\begin{eqnarray*}
\lim_{m \rightarrow \infty} \frac{\alpha(m)}{m} &=&  \lim_{m \rightarrow \infty} \frac{\big \lfloor -\frac{1}{2} + \sqrt{\frac{1}{4} + rm^2+rm} \big \rfloor}{m}\\
&=& \sqrt{r}
\end{eqnarray*}
so the $x$ and $y$ intercepts of the limiting shape $P$ are both equal to $\sqrt{r}$.  Since $\sqrt{r} \cdot \sqrt{r} = r$, Corollary \ref{cor:InterceptsDetermineShape} tells us that the boundary of $P$ is defined by the line through the $x$- and the $y$-intercepts as claimed.
\end{proof}

\section{The Symbolic Generic Initial System of 6, 7, and 8 Uniform Fat Points in General Position}
\label{sec:678}


As before, $I \subseteq R[x,y,z]$ will denote the ideal of points $p_1, \dots, p_r$ of $\mathbb{P}^2$ in general position.  The goal of this section is to prove the second part of Theorem~\ref{thm:mainthm}.  

\newtheorem*{thm:mainthmb}{Theorem \ref{thm:mainthm} (b)}
\begin{thm:mainthmb}
\label{thm:polytopefor678}
Suppose that $I \subseteq K[x,y,z]$ is the ideal of $r=6, 7, \text{or } 8$ points of $\mathbb{P}^2$ in general position and that $P \subseteq \mathbb{R}^2$ is the limiting shape of the reverse lexicographic symbolic generic initial system $\{ \text{gin} (I^{(m)})\}_m$.  Then the boundary of $P$ is defined by the line segment through the points $(\gamma_1, 0)$ and $(0, \gamma_2)$ where 
\begin{itemize}
\item[(a)] $\gamma_1 = \frac{12}{5}$ and $\gamma_2 = \frac{5}{2}$ when $r=6$;
\item[(b)]  $\gamma_1 = \frac{21}{8}$ and $\gamma_2 = \frac{8}{3}$ when $r=7$; and
\item[(c)] $\gamma_1 = \frac{48}{17}$ and $\gamma_2 = \frac{17}{6}$ when $r=8$.
\end{itemize}
\end{thm:mainthmb}

The proof of this result relies on knowing certain values of the Hilbert functions $H_{I^{(m)}}(t)$ of the ideals $I^{(m)}$ where $I$ is the ideal of 6, 7, or 8 general points.  Techniques for computing $H_{I^{(m)}}(t)$ in these cases are not new (for example, see \cite{Nagata60}), but they can be complicated.  Thus, we take time in Section \ref{sec:678Background} to review a modern technique for finding the Hilbert functions, and then apply these results to the proof of Theorem \ref{thm:mainthm}(b) in Section \ref{sec:678Proof}.

\subsection{Background on Surfaces}
\label{sec:678Background}
The method we use to compute $H_{I^{(m)}}(t)$ follows the work of Fichett, Harbourne, and Holay in \cite{FHH01}.

Suppose that $\pi: X \rightarrow \mathbb{P}^2$ is the blow-up of distinct points $p_1, \dots, p_r$ of $\mathbb{P}^2$. Let $E_i = \pi^{-1}(p_i)$ for $i = 1, \dots, r$ and $L$ be the total transform in $X$ of a line  not passing through any of the points $p_1, \dots, p_r$.  The classes of these divisors form a basis of $\text{Cl}(X)$; for convenience, we will write $e_i$ for the class $[E_i]$ of $E_i$ and $e_0$ for the class $[L]$.  Further, the intersection product in $\text{Cl}(X)$ is defined by $e_i^2 = -1$ for $i=1, \dots, r$; $e_0^2 = 1$; and $e_i \cdot e_j = 0$ for all $i\neq j$.

Let $Z_m = m(p_1+\cdots +p_r)$ be a uniform fat point subscheme with sheaf of ideals $\mathcal{I}_{Z_m}$; set 
$${F}_d = dE_0 - m(E_1 + E_2 + \cdots +E_r)$$
and $\mathcal{F}_d = \mathcal{O}_X(F_d)$.  Then $\pi_{*}(\mathcal{F}_d) = \mathcal{I}_Z \otimes \mathcal{O}_{\mathbb{P}^2}(d)$ so 
$$\text{dim}((I_{Z_m})_d) = h^0(\mathbb{P}^2, \mathcal{I}_Z \otimes \mathcal{O}_{\mathbb{P}^2}(d)) = h^0(X, \mathcal{F}_d)$$
for all $d$.  In particular, if $I \subseteq K[x,y,z]$ is the ideal of the points $p_1, \dots, p_r$ in $\mathbb{P}^2$,
$$H_{I^{(m)}}(d) = h^0(X, \mathcal{F}_d)$$
and so we can study the Hilbert function of the symbolic powers $I^{(m)}$ by working with divisors on the surface $X$.   For convenience, we will often write $h^0(X, F) = h^0(X, \mathcal{O}_X(F))$.


Recall that if  $[{F}]$ not the class of an effective divisor then $h^0(X, {F}) = 0$.  On the other hand, if $F$ is effective, then we will see that we can compute $h^0(X,{F})$ by computing $h^0(X,{H})$ for some \textit{numerically effective} divisor $H$.  

\begin{defn}
A divisor $H$ is \textbf{numerically effective} if $[F] \cdot [H] \geq 0$ for every effective divisor $F$, where $[F] \cdot [H]$ denotes the intersection multiplicity.  The cone of classes of numerically effective divisors in $\text{Cl}(X)$ is denoted by NEF($X$).
\end{defn}

\begin{lem}
\label{lem:h0ofNEFF}
Suppose that $X$ is the blow-up of $\mathbb{P}^2$ at $r \leq 8$ points in general position and that $F \in \text{NEF}(X)$.  Then $F$ is effective and 
$$h^0(X, F)  = ([F]^2-[F]\cdot [K_X])/2+1$$
where $K_X = -3e_0 + e_1 + \cdots + e_r$.
\end{lem}

\begin{proof}
This is a consequence of Riemann-Roch and the fact that $h^1(X, F) = 0$ for any numerically effective divisor $F$. See Theorem 8 of \cite{Harbourne96} or Section 1 of \cite{FHH01} for a discussion.
\end{proof}

\begin{cor}
\label{cor:h0ofFtNEFF}
Let $F_t = tL - m(E_1+E_2 + \cdots + E_r)$.  If $F_t$ is numerically effective then 
$$h^0(X,{F}_t) = {t+2 \choose 2}-r{m+1 \choose 2}.$$
\end{cor}

A divisor class $[C]$ on $X$ is said to be \textit{exceptional} if it is the class of an exceptional divisor $C$ on $X$ (that is, a smooth curve isomorphic to $\mathbb{P}^1$ such that $[C]^2 = -1$).\footnote{Note that if $[C]$ is an exceptional class, there is a unique effective divisor in this class, typically called the \textit{exceptional curve}}  The following result of Fichett, Harbourne, and Holay \cite{FHH01} tells us how to detect if a divisor is numerically effective if we know the exceptional curves.  

\begin{lem}[{Lemma 4(b) of \cite{FHH01}}]
\label{lem:IdentifyingNEFF}
Suppose that $X$ is the surface obtained by blowing up $2 \leq r \leq 8$ points of $\mathbb{P}^2$.  Then $F$ is numerically effective if the intersection multiplicity of $[F]$ with all exceptional classes  is greater than or equal to 0.
\end{lem}

Another result from \cite{FHH01} tells us what the exceptional curves of $X$ are in the cases that we are interested in.

\begin{lem}[{Lemma 3(a) of \cite{FHH01}}]
\label{lem:exceptionalcurves}
Let $C$ be a curve on the blow-up $X$ of $\mathbb{P}^2$ at 8 points in general position.  Then, with the notation above, the exceptional classes are the following, up to permutation of indices $1, 2, \dots, 8$: 
\begin{multicols}{2}
\begin{itemize}
\item  $h_1=e_8$
\item $h_2=e_0-e_1-e_2$
\item $h_3=2e_0-e_1-\cdots -e_5$
\item $h_4=3e_0-2e_1-e_2-\cdots -e_7$
\item $h_5=4e_0-2e_1-2e_2-3e_3-e_4-\cdots -e_8$
\item $h_6=5e_0-2e_1-\cdots -2e_6-e_7-e_8$
\item $h_7=6e_0-3e_1-2e_2-\cdots -2e_8$.
\end{itemize}
\end{multicols}
When $X$ is the blow-up of $\mathbb{P}^2$ at $n \leq 8$ points, the exceptional classes of $X$ are the ones listed above with $8-n$ of the $e_i$ ($i=1, \dots, 8$) set to 0.
\end{lem}

It turns out that knowing how to compute $h^0(X, H)$ for a numerically effective divisor $H$ will actually allow us to compute $h^0(X, F)$ for \textit{any} divisor $F$.  In particular, for any divisor $F$, there exists a divisor $H$ such that $h^0(X, F) = h^0(X, H)$ and either: 
\begin{enumerate}
\item[(a)]  $H$ is numerically effective so $$h^0(X, F) = h^0(X, H) = (H^2-H\cdot K_X)/2+1$$ by Lemma \ref{lem:h0ofNEFF}; or
\item[(b)]  There is a numerically effective divisor $G$ such that $[G]\cdot [H] <0$ so $[H]$ is not the class of an effective divisor and $h^0(X, F) = h^0(X, H) = 0$.
\end{enumerate}

The following result will be used in Procedure \ref{proc:findH} to find such an $H$.

\begin{lem}
\label{lem:Equalh0}
Suppose that $[C]$ is an exceptional class such that $[F] \cdot [C] <0$.  Then $h^0(X, F) = h^0(X, F-C)$.
\end{lem}

\begin{proof}
Note that it suffices to prove this statement for the case where $C$ is a smooth curve isomorphic to $\mathbb{P}^1$:  if $[C']$ is an exceptional class then there exists a smooth curve $C$ isomorphic to $\mathbb{P}^1$ such that $[C'] = [C]$ so $[C'] \cdot [F] = [C] \cdot [F]$ and $h^0(X, F-C') = h^0(X, F-C)$. Note that we have an exact sequence 
$$ \mathcal{O}_X(F-C) \rightarrow \mathcal{O}_X(F) \rightarrow \mathcal{O}_C(F) \cong \mathcal{O}_{\mathbb{P}^1}([F] \cdot [C]) \rightarrow 0$$
induced by tensoring the exact sequence 
$$0 \rightarrow \mathcal{O}_X(-C) \rightarrow \mathcal{O}_X \rightarrow \mathcal{O}_C \rightarrow 0$$
with $\mathcal{O}_X(F)$.  Then, from the long exact sequence of cohomology, $h^0(X, \mathcal{O}_X(F-C)) = h^0(X, \mathcal{O}_X(F))$ since $h^0(X, \mathcal{O}_{\mathbb{P}^1}([F]\cdot[C])) = 0$ ($[F]\cdot[C] <0$).
\end{proof}

The method for finding such the $H$ described above is as follows.

\begin{procedure}
\label{proc:findH}
Given a divisor $F$ we can find a divisor $H$ with $h^0(X, F) = h^0(X, H)$ satisfying either condition (a) or (b) above as follows.
\begin{enumerate}
\item Reduce to the case where $[F] \cdot e_i \geq 0$ for all $i=1, \dots, n$:  if $[F]\cdot e_i <0$ for some $i$, $h^0(X, F) = h^0(X, F-([F]\cdot e_i)E_i)$, so we can replace $F$ with $F - ([F]\cdot e_i) E_i$.
\item Since $L$ is numerically effective, if $[F]\cdot e_0<0$ then $[F]$ is not the class of an effective divisor and we can take $H=F$ (case (b)).
\item If $[F] \cdot [C] \geq 0$ for every exceptional class $[C]$ then, by Lemma \ref{lem:IdentifyingNEFF}, $F$ is numerically effective, so we can take $H = F$ (case (a)).
\item If $[F] \cdot [C] <0$ for some exceptional class $[C]$ then $h^0(X, F) = h^0(X, F-C)$ by Lemma \ref{lem:Equalh0}. Then replace $F$ with $F-C$ and repeat from Step 2.  
\end{enumerate}
\end{procedure}

There are only a finite number of exceptional classes to check by Lemma \ref{lem:exceptionalcurves} so it is possible to complete Step 3.  Further, it is easy to see with Lemma \ref{lem:exceptionalcurves} that $F \cdot e_0 > [F-C] \cdot e_0$ when $[C]$ is an exceptional curve, so the condition in Step 2 will be satisfied after at most $[F]\cdot e_0 +1$ repetitions.  Thus, this process will eventually terminate.\footnote{The decomposition $F = H+(F-H)$ has been referred to as a \textit{Zariski decomposition} in some of the literature on fat points (for example, in \cite{FHH01}), but we avoid this terminology here because it is not consistent with definitions elsewhere (for example, in \cite{Lazarsfeld04}).}

\subsection{Proof of Theorem \ref{thm:mainthm}(b)}
\label{sec:678Proof}


The proof of each part of Theorem \ref{thm:mainthm}(b) follows the same five steps outlined below. In Step 4, we will use the following lemma.

\begin{lem}
\label{lem:moninz}
Let $I$ be the ideal of $r$ points in $\mathbb{P}^2$.  The number of monomials in $\text{gin}(I^{(m)}) \subseteq K[x,y,z]$ of degree $t$ involving the variable $z$ is equal to $H_{I}(t-1)$.
\end{lem}

\begin{proof}
Since, by Proposition \ref{prop:geninxy}, $\text{gin}(I^{(m)})$ is generated in the variables $x$ and $y$, the only elements of $\text{gin}(I^{(m)})_t$ that involve $z$ have to arise by multiplying monomials of $\text{gin}(I^{(m)})_{t-1}$ by $z$.   Since multiplying each of the $H_{I^{(m)}}(t-1)$ monomials in $\text{gin}(I^{(m)})_{t-1}$ by $z$ gives distinct monomials, the result follows.
\end{proof}

As in Section \ref{sec:678Background}, $Z_m = m(p_1+\cdots +p_r)$ is a uniform fat point subscheme supported at $r$ distinct general points $p_1, \dots, p_r$ and $I$ is the ideal of $p_1, \dots, p_r$ so that $I^{(m)} = I_{Z_m}$.  Recall that if $F_t = tL - m(E_1+\cdots +E_r)$ then
$$H_{I_{Z_m}}(t) = h^0(X, \mathcal{F}_t);$$
we also write $H_{I_{Z_m}}(t) = H_Z(t)$.
Finally, 
$$\alpha(m) := \min \{ t : H_{I_{Z_m}}(t) \neq 0\}.$$

\textbf{Step 1:}  Find the smallest $N$ such that ${F}_t = tE_0-m(E_1+\cdots+E_r)$ is numerically effective for all $t \geq N$.  To do this we will find the smallest $N$ such that  $[F_t] \cdot [C] \geq 0$ for all $t \geq N$ (see Lemma \ref{lem:IdentifyingNEFF}).  By Corollary \ref{cor:h0ofFtNEFF},
 $$h^0(X, \mathcal{F}_t) = {t+2 \choose 2}-r{m+1 \choose 2}$$
 for all $t \geq N$.

\textbf{Step 2:}  Use some optimal numerically effective divisor $D$ to find $M$ such that $[{F}_t]\cdot [D] <0$ for all $t<M$.  By the definition of a numerically effective divisor, this will show that $[F_t]$ for $t<M$ is not the class of an effective divisor, and thus that $H_{I_{Z_m}}(t) = h^0(X, \mathcal{F}_t) = 0$ for all $t<M$.

\textbf{Step 3:}  Show that $h^0(X, \mathcal{F}_M) \neq 0$ where $M$ is as in Step 2.  To do this, we will
use Procedure \ref{proc:findH} to find a numerically effective $H$ such that $h^0(X, F_M) = h^0(X, H_M)$.  Together with Step 2, this will show that $\alpha(m) = M$, and $x^M$ is the smallest power of $x$ in $\text{gin}(I^{(m)})$.

\textbf{Step 4:}  By Lemma \ref{lem:moninz}, the number of monomials of degree $t$ in the $\text{gin}(I_{Z_m})$ involving only the variables $x$ and $y$ is equal to $H_Z(t) -H_Z(t-1)$.  Using this, show that the number of monomials in $\text{gin}(I_{Z_m})$ of degree $N+1$ in $x$ and $y$ is exactly $N+2$ (here $N$ is as in Step 1).  This implies that all monomials in $x$ and $y$ of degree $N+1$ are in the $\text{gin}(I_{Z_m})$.  

Use Lemma \ref{lem:moninz} again to show that the number of monomials in $\text{gin}(I_{Z_m})$ of degree $N$ involving only $x$ and $y$ is strictly less than $N+1$, so not all monomials of degree $N$ in $x$ and $y$ are in $\text{gin}(I^{(m)})$.  Since the ideals of the symbolic generic initial system are generated in $x$ and $y$ (Proposition \ref{prop:geninxy}), this will imply that $y^{N+1}$ is the smallest power of $y$ in $\text{gin}(I^{(m)})$.

\textbf{Step 5:}  The smallest power of $y$ in $\text{gin}(I^{(m)})$ is $N+1$ by Step 4 and the smallest power of $x$ in $\text{gin}(I^{(m)})$ is $\alpha(m)=M$ by Step 3.
Thus, the intercepts of the limiting shape $P$ of the symbolic generic initial system of $I$ are $\big( 0, \lim_{m \rightarrow \infty} \frac{N+1}{m} \big)$ and $\big( \lim_{m \rightarrow \infty} \frac{M}{m}, 0 \big)$.
Since
$$ \Big( \lim_{m \rightarrow \infty} \frac{N+1}{m}\Big) \cdot \Big( \lim_{m \rightarrow \infty} \frac{M}{m} \Big) = r,$$
Corollary \ref{cor:InterceptsDetermineShape} implies that the limiting shape $P$ is as claimed in Theorem \ref{thm:mainthm}(b).

\subsubsection{6 General Points}

Throughout this section $I$ is the ideal of 6 points $p_1, \dots, p_6$ of $\mathbb{P}^2$ in general position and $Z_m = m(p_1 + \cdots + p_6)$ so $I_{Z_m} = I^{(m)}$.  The exceptional classes of the blow-up $X$ of $\mathbb{P}^2$ at $p_1, \dots, p_6$ are those in Lemma \ref{lem:exceptionalcurves} with two of the $e_i$ set to 0. 

\textbf{Step 1:} To find an $N$ such that $F_t = tL-m(E_1+\cdots+E_6)$ is numerically effective for all $t \geq N$ we will use the permutation of the exceptional curves from Lemma \ref{lem:exceptionalcurves} that is most likely to make $h_i \cdot [F_t]$ negative.  

\begin{eqnarray*}
F_t \cdot h_2= t - 2m \geq 0 &\iff& t \geq 2m\\
F_t\cdot h_3 = 2t - 5m \geq 0 &\iff& t \geq \frac{5}{2}m\\
F_t\cdot h_4 = 3t - 2m -5m\geq 0 &\iff& t \geq \frac{7}{3}m\\
F_t\cdot h_5 = 4t - 3\cdot 2m - 3m \geq 0 &\iff& t \geq \frac{9}{4}m\\
F_t \cdot h_6 = 5t - 2\cdot 6m \geq 0 &\iff& t \geq \frac{12}{5}m\\
F_t\cdot h_7 = 6t - 3m - 2\cdot 5 m \geq 0 &\iff& t \geq \frac{13}{6}m\\
\end{eqnarray*}

The strongest condition on $t$ is $t \geq \frac{5}{2}m$.  Thus, $N = \frac{5}{2}m$ and $F_t$ is numerically effective for all $t \geq \frac{5}{2}m$.  Thus, 
$$H_{I^{(m)}}(t) = h^0(X, \mathcal{F}_t) = {t+2 \choose 2}-6{m+1 \choose 2}$$
for all $t \geq \frac{5}{2}m$.

\textbf{Step 2:}  Now we want to find an optimal numerically effective divisor $D$ such that $[{F}_t] \cdot [D]<0$ for small $t$.  By the calculations in Step 1,
$$D = 5L-2(E_1+\cdots +E_6)$$
is numerically effective ($D = F_5$ when $m=2$).

If $[F_t]$ is the class of an effective divisor then $[D] \cdot[{F}_t]  \geq 0$.  Thus, if $[D] \cdot [{F}_t] <0$ then $[F_t]$ is not effective.  Note that 
$$[D] \cdot [{F}_t] = 5t-2\cdot6m <0 \iff t< \frac{12m}{5}.$$
Thus, $[{F}_t]$ is not the class of an effective divisor when $t< \frac{12m}{5}$ so $h^0(X, \mathcal{F}_t) = 0$ for $ t< \frac{12}{5}m$.  We set $M = \frac{12}{5}m$.

\textbf{Step 3:}  Starting with this step we will make a divisibility assumption on $m$. Suppose that $m$ is divisible by both 5 and 2, so 
$$m = 10m'$$
for some integer $m'$.  The goal of this step is to show that $F_{\frac{12}{5}m} = F_{24m'}$ is in the class of an effective divisor; to do this we follow Procedure \ref{proc:findH}.  One can check that the only exceptional class that has a negative intersection multiplicity with $[F_{24m'}]$ is $h_3=[2L-E_1- \cdots -E_5]$:
$$[{F}_{24m'}] \cdot h_3 = 2 \cdot 24 m' - 5\cdot 10m' = -2m'.$$

At this point it will be useful to distinguish between the permutations of  $h_3$; we will denote $[2L-E_1 - \cdots -\widehat{E_i} - \cdots - E_6]$ by $h_{3_i}$.

\begin{lem}
\label{lem:ifonethenall6}
Let $F_t = tL-m(E_1+\cdots +E_6)$.  If $[{F}_t] \cdot h_{3_1} <0$ then $([{F}_t] - h_{3_1} - \cdots - h_{3_i}) \cdot h_{3_{i+1}} <0$ for $i=1, \dots, 5$.  
\end{lem}

\begin{proof}
Suppose that $[{F}_t] \cdot h_{3_1}<0$, or equivalently, that $t< \frac{5}{2}m$.  Then 
$$[{F}_t] - h_{3_1} - \cdots - h_{3_i}= (t-2i)e_0 - (m-(i-1))(e_1+\cdots +e_i) - (m-i)(e_{i+1} + \cdots + e_6)$$
so 
$$([{F}_t] - h_{3_1} - \cdots - h_{3_i})\cdot h_{3_{i+1}} = 2(t-2i) - (m-i+1)(i) - (m-i)(6-i-1) = 2t - 5m$$
which is less than 0 since $t< \frac{5}{2}m$.
\end{proof}

We will denote the sum $h_{3_1} + \cdots +h_{3_5}$ by $[Y]$ and call it a \textit{cycle}.   Note that 
$$[Y]=(2\cdot 6)e_0 - 5(e_1+\cdots +e_6).$$
By Lemma \ref{lem:ifonethenall6}, if $[{F}_t] \cdot h_3 <0$ for one permutation then we subtract an entire cycle from $[F_t]$ when following Procedure \ref{proc:findH}.  

When following Procedure \ref{proc:findH}, we subtract off $2m'$ full cycles from $[{F}_{24m'}]$;
\begin{eqnarray*}
[{F}_{24m'}] - 2m'[Y] &=& (24m' -12 \cdot 2m')e_0 - (10m'-5\cdot 2m')(e_1+\cdots +e_6) \\
&=& 0e_0-0(e_1+\cdots +e_6)
\end{eqnarray*}
so $H_{24m'} = 0$.
Therefore, $H_{I^{(m)}}(24m') = h^0(X, \mathcal{F}_{24m'}) = h^0(X, 0) = 1$ and $\alpha(m) = \frac{12m}{5}$ when $m$ is divisible by 10.

\textbf{Step 4:}  Again, in this section we will assume that 10 divides $m$ and we write $m = 10m'$ for some integer $m'$.  Then $N=\frac{5}{2}m = 25m'$.  By Lemma \ref{lem:moninz}, there are $H_Z(N+1) - H_Z(N)$ monomials in $\text{gin}(I^{(m)})$ that involve only $x$ and $y$. Since $F_N$ and $F_{N+1}$ are numerically effective by Step 1,
\begin{eqnarray*}
H_Z(N+1) - H_Z(N) &=& h^0(X,\mathcal{F}_{N+1}) - h^0(X,\mathcal{F}_{N})\\
&=& {N+2 \choose 1} = N+2.
\end{eqnarray*}
Thus, $\text{gin}(I_{Z_m})_{N+1}$ contains all monomials of degree $N+1$ in the variables $x$ and $y$.

Now we need to determine $H_Z(N) - H_Z(N-1)$ and show that it is less than $N+1$ (that is, $\text{gin}(I_{Z_m})_N$ does not contain all monomials in $x$ and $y$ of degree $N$).  Consider 
$$F_{N-1} = (25m'-1)L - 10m'(E_1+\cdots +E_6).$$
Then, following Procedure \ref{proc:findH}, we can subtract exactly 2 cycles $[Y]$ from $[{F}_{N-1}]$ to obtain $[H_{N-1}]$.  We get
$$[{H}_{N-1}] = (25m'-1-24)e_0 - (10m'-10)(e_1+ \cdots + e_6)$$
so, by Corollary \ref{cor:h0ofFtNEFF},
\begin{eqnarray*}
h^0(X, \mathcal{F}_{N-1}) &=& h^0(X, \mathcal{H}_{N-1})\\
&=& \frac{25}{2}m'^2-\frac{35}{2}m'+6.
\end{eqnarray*}

By Step 1, $F_N$ is numerically effective so, again by Corollary \ref{cor:h0ofFtNEFF},
$$h^0(X, \mathcal{F}_N) = \frac{25}{2}m'^2+\frac{15}{2}m'+1$$

Thus, 
$$H_{I_{Z_m}}(N)-H_{I_{Z_m}}(N-1) = 25m'-5<N+1 = 25m'+1$$
and not all monomials in $x$ and $y$ of degree $N$ are contained in $\text{gin}(I^{(m)})$.  Therefore, the largest degree generator of $\text{gin}(I_{Z_m})$ is of degree $N+1 = \frac{5}{2}m+1$ when $m$ is divisible by 10.

\textbf{Step 5:}  By Step 4, the highest degree generator of $\text{gin}(I^{(m)})$ is of degree $\frac{5}{2}m+1$ when $m$ is divisible by 10.  By Step 3, the smallest degree element of $\text{gin}(I^{(m)})$ is of degree $\alpha(m) = \frac{12m}{5}$ when $m$ is divisible by 10.  Thus, the intercepts of the limiting shape of the symbolic generic initial system of $I$ are $(0, \frac{5}{2})$ and $(\frac{12}{5},0)$.  Since 
$$\frac{12}{5} \cdot \frac{5}{2} = 6,$$
Corollary \ref{cor:InterceptsDetermineShape} tells us that the boundary of the limiting shape is defined by the line through the intercepts and is as claimed in Theorem \ref{thm:mainthm}(b).

\subsubsection{7 General Points}

Throughout this section $I$ is the ideal of 7 points $p_1, \dots, p_7$ of $\mathbb{P}^2$ in general position and $Z = m(p_1 + \cdots + p_7)$ so $I_{Z_m} = I^{(m)}$.  The exceptional classes of the blow-up $X$ of $\mathbb{P}^2$ at $p_1, \dots, p_7$ are those in Lemma \ref{lem:exceptionalcurves} with one of the $e_i$ set to 0. 

\textbf{Step 1:} To find an $N$ such that $F_t = tL-m(E_1+\cdots+E_7)$ is numerically effective for all $t \geq N$ we will use the permutation of the exceptional curves from Lemma \ref{lem:exceptionalcurves} that is most likely to make $h_i \cdot [F_t]$ negative.  Similar to the case of six points, the strongest condition on $t$ from $h_i \cdot [F_t] \geq 0$  is $t \geq \frac{8}{3}m$.  Thus, $N = \frac{8}{3}m$ and $F_t$ is numerically effective for all $t \geq N$.  Further, 
$$H_{I^{(m)}}(t) = h^0(X, F_t) = {t+2 \choose 2}-7{m+1 \choose 2}$$
for all $t \geq \frac{8}{3}m$.

\textbf{Step 2:}  Now we want to find an optimal numerically effective divisor $D$.  By the calculations in Step 1,
$$D = 8L-3(E_1+\cdots +E_7)$$
is numerically effective ($D = F_8$ when $m=3$).

If $[F_t]$ is the class of an effective divisor then $[D] \cdot [{F}_t]  \geq 0$.  We want to know when $[D] \cdot [{F}_t]$ is strictly less than 0 because this will imply that $[F_t]$ is not the class of an effective divisor.  Note that 
$$[D] \cdot [{F}_t] = 8t-3\cdot7m <0 \iff t< \frac{21m}{8}.$$
Thus, $h^0(X, \mathcal{F}_t) = 0$ for $ t< \frac{21}{8}m$.  We set $M = \frac{21}{8}m$.

Our next goal is to show that this is an optimal value.  That is, if $\frac{21}{8}m$ is an integer, then $h^0(X, \mathcal{F}_{\frac{21}{8}m}) \neq 0$.

\textbf{Step 3:}  Starting with this step we will make a divisibility assumption on $m$ and suppose that $m$ is divisible by both 8 and 3, so 
$$m = 24m'$$
for some integer $m'$.  The goal of this step is to show that $F_{\frac{21}{8}m} = F_{21\cdot 3m'}$ is in the class of an effective divisor; to do this we will follow Procedure \ref{proc:findH}.  One can check  that the only exceptional class which has a negative intersection multiplicity with $[F_{63m'}]$ is $ h_4=[3L-2E_1- E_2-\cdots -E_7]$:
$$[F_{21\cdot 3m'}] \cdot h_4 = 3 \cdot 63 m' -2\cdot 24m'-6\cdot 24m' = -3m'.$$

At this point it will be useful to distinguish between the permutations of  $h_4$;  we will denote $[3L-2E_i - E_1 - \cdots -\widehat{E_i} - \cdots - E_7]$ by $h_{4_i}$.

\begin{lem}
\label{lem:ifonethenall7}
Let $F_t = tL-m(E_1+\cdots +E_7)$.  If $[{F}_t] \cdot h_{4_1} <0$ then $([{F}_t] - h_{4_1} - \cdots - h_{4_i}) \cdot h_{4_{i+1}} <0$ for $i=1, \dots, 6$.  
\end{lem}


We will denote the sum of all seven permutations of $h_4$ by $[Y]$ and call it one \textit{cycle}.   Note that 
$$[Y]=21e_0 - 8(e_1+\cdots +e_7).$$

It is easy to see that we can subtract off $3m'$ full cycles from $[{F}_{21\cdot3m'}]$ to obtain $[H_{21\cdot3m'}]$.  We have
\begin{eqnarray*}
[{F}_{21\cdot 3m'}] - 3m'[Y] &=& (21\cdot 3m' -21 \cdot 3m')e_0 - (24m'-8\cdot 3m')(e_1+\cdots +e_7)\\
 &=& 0L-0(e_1+\cdots +e_7).
\end{eqnarray*}
so $H_{21\cdot3m'} = 0$ and $h_{I_{Z_m}}(21\cdot 3m') = h^0(X, F_{21\cdot 3m'}) = h^0(X, 0) = 1$ and $\alpha(m) = 21\cdot 3m' = \frac{21m}{8}$ when $m$ is divisible by $8\cdot3$.

\textbf{Step 4:}  Again, in this section we will assume that 24 divides $m$, so we write $m = 24m'$ for some integer $m'$.  Then $N= \frac{8}{3}m = 8^2m'$.
We now want to show that the number of monomials in the variables $x$ and $y$ in $\text{gin}(I^{(m)})_{N+1}$ is equal to $N+2$ and that the number of monomials of degree $N$ in $x$ and $y$ in $\text{gin}(I^{(m)})_N$ is less than $N+1$.  This will prove that the highest degree generator occurs in degree $N+1$.

By Lemma \ref{lem:moninz}, there are $H_Z(N+1) - H_Z(N)$ monomials in $\text{gin}(I^{(m)})$ that involve only $x$ and $y$.  Then, since $F_N$ and $F_{N+1}$ are numerically effective by Step 1,
$$H_Z(N+1) - H_Z(N) = {N+2 \choose 1} = N+2.$$
Thus, $\text{gin}(I_{Z_m})_{N+1}$ contains all monomials of degrees $N+1$ in the variables $x$ and $y$.

Now we need to determine $H_Z(N) - H_Z(N-1)$ and show that it is less than $N+1$ (that is, $\text{gin}(I_{Z_m})_N$ does not contain all monomials in $x$ and $y$ of degree $N$).  

Consider 
$$F_{N-1} = (8^2m'-1)L - 24m'(E_1+\cdots +E_7).$$
Recall from Step 3 that one cycle is equal to $[Y]=3\cdot 7e_0 - 8(e_1+\cdots +e_7).$  By Procedure \ref{proc:findH},
$$[{H}_{N-1}] = (8^2m'-1-63)e_0 - (24m'-24)(e_1+ \cdots + e_7)$$
so, by Corollary \ref{cor:h0ofFtNEFF},
$$h^0(X, \mathcal{F}_{N-1}) =  h^0(X, \mathcal{H}_{N-1}) = 32m'^2-52m'+21.$$

From Step 1 we know that ${F}_N$ is numerically effective and
$$h^0(X, \mathcal{F}_N)  = 32m'^2+12m'+1$$

Thus, 
$$H_{I_{Z_m}}(N)-H_{I_{Z_m}}(N-1) = 64m'-20<N+1 = 64m'+1$$
and not all monomials in $x$ and $y$ of degree $N$ are contained in $\text{gin}(I^{(m)})$.  Therefore, the largest degree generator of $\text{gin}(I_{Z_m})$ is of degree $N+1 = \frac{8}{3}m+1$ when $m$ is divisible by 24.

\textbf{Step 5:}  By Step 4, the highest degree generator of $\text{gin}(I^{(m)})$ is of degree $\frac{8}{3}m+1$ when $m$ is divisible by 24.  By Step 3, the smallest degree element of $\text{gin}(I^{(m)})$ is of degree $\alpha(m) = \frac{21m}{8}$ when $m$ is divisible by 24.  Thus, the intercepts of the limiting shape of the symbolic generic initial system of $I$ are $(0, \frac{8}{3})$ and $(\frac{21}{8},0)$.  Since
$$\frac{8}{3} \cdot \frac{21}{8} = 7,$$
Corollary \ref{cor:InterceptsDetermineShape} tells us that the limiting polytope is defined by the line through the intercepts and is as claimed in Theorem \ref{thm:polytopefor678}.

\subsubsection{8 General Points}

Throughout this section $I$ is the ideal of 8 points $p_1, \dots, p_8$ of $\mathbb{P}^2$ in general position and $Z_m = m(p_1 + \cdots + p_8)$ so $I_{Z_m} = I^{(m)}$.  The exceptional classes of the blow-up $X$ of $\mathbb{P}^2$ at $p_1, \dots, p_8$ are those in Lemma \ref{lem:exceptionalcurves}.

\textbf{Step 1:} To find an $N$ such that $F_t = tL-m(E_1+\cdots+E_8)$ is numerically effective for all $t \geq N$ we compute $[C] \cdot [{F}_t]$ for all exceptional classes $[C]$ of $X$.  Similar to the case of six points, we can check that the strongest condition on $t$ in resulting from $[C] \cdot [F_t] \geq 0$ is $t \geq \frac{17}{6}m$.  Thus, $N = \frac{17}{6}m$ and $F_t$ is numerically effective for all $t \geq N$.  Further, 
$$h^0(X, \mathcal{F}_t) = {t+2 \choose 2}-8{m+1 \choose 2}$$
for all $t \geq \frac{17}{6}m$.

\textbf{Step 2:}  Now we want to find an optimal numerically effective divisor $D$.  By the calculations in Step 1, 
$$D = 17L-6(E_1+\cdots +E_8)$$
is numerically effective ($D = F_{17}$ when $m=6$).

If $[F_t]$ is the class of an effective divisor then $[D] \cdot [{F}_t]  \geq 0$.  We want to know when $[D] \cdot [{F}_t]$ is strictly less than 0 which will imply that $[F_t]$ is not the class of an effective divisor.  Note that 
$$[D] \cdot [{F}_t] = 17t-6\cdot8m <0 \iff t< \frac{48m}{17}.$$
Thus, $h^0(X, \mathcal{F}_t) = 0$ for $ t< \frac{48}{17}m$.  We set $M = \frac{48}{17}m$.

Our next goal is to show that this is an optimal value.  That is, if $\frac{48}{17}m$ is an integer, then $h^0(X, \mathcal{F}_{\frac{48}{17}m}) \neq 0$.

\textbf{Step 3:}  Starting with this step we will make a divisibility assumption on $m$ and suppose that $m$ is divisible by both 17 and 6, so 
$$m = 17 \cdot 6m'$$
for some integer $m'$.  The goal of this step is to show that $F_{\frac{48}{17}m} = F_{48\cdot 6m'}$ is in the class of an effective divisor.  To do this we will follow Procedure \ref{proc:findH} to find $H_{48\cdot 5m'}$.  One can check that the only exceptional class that has a negative intersection multiplicity with $[\mathcal{F}_{48 \cdot 6m'}]$ is $h_7=[6L-3E_1- 2E_2-\cdots -2E_8]$:
$$[{F}_{48\cdot 6m'}] \cdot h_7 = 6 \cdot 48\cdot 6 m' -3\cdot 17\cdot 6m' - 2\cdot 7 \cdot 17\cdot 6 m' = -6m'.$$

It will be useful to distinguish between the permutations of  $h_7$.  We will denote $[6L-3E_i - 2E_1 - \cdots -\widehat{2E_i} - \cdots - 2E_7]$ by $h_{7_i}$.

\begin{lem}
\label{lem:ifonethenall8}
Let $F_t = tL-m(E_1+\cdots +E_8)$.  If $[{F}_t] \cdot h_{7_1} <0$ then $([{F}_t] - h_{7_1}- \cdots -h_{7_i}) \cdot h_{7_{i+1}} <0$ for $i=1, \dots, 7$.  
\end{lem}


We will denote the sum of all eight permutations of $h_7$ by $[Y]$ and call it a \textit{cycle}.   Note that 
$$Y=48e_0 - 17(e_1+\cdots +e_8).$$

Following Procedure \ref{proc:findH}, we subtract off $6m'$ full cycles from $[{F}_{48\cdot6m'}]$ to get $H_{48\cdot 6m'}$.  Then
\begin{eqnarray*}
[F_{48\cdot 6m'}] - 6m'[Y] &=& (48\cdot 6m' -48 \cdot 6m')e_0 - (17\cdot 6m'-17\cdot 6m')(e_1+\cdots +e_8)\\
 &=& 0e_0-0(e_1+\cdots +e_7).
\end{eqnarray*} 

Therefore, $h^0(X, \mathcal{F}_{48\cdot 6m'}) = h^0(X, 0) = 1$ and $\alpha(m) = 48\cdot 6m' = \frac{48m}{17}$ in this case.

\textbf{Step 4:}  Again, in this section we will assume that 6 and 17 divide $m$, so $m = 17 \cdot 6m'$ for some integer $m'$ and $N = \frac{17}{6}m= 17^2m'$.
We now want to show that the number of monomials in only $x$ and $y$ in $\text{gin}(I^{(m)})_{N+1}$ is $N+2$ and that the number of monomials of degree $N$ in the variables $x$ and $y$ in $\text{gin}(I^{(m)})$ is less than $N+1$.  This will show that the highest degree generator occurs in degree $N+1$.

By Lemma \ref{lem:moninz}, there are $H_Z(N+1) -H_Z(N)$ monomials in $\text{gin}(I^{(m)})$ that involve only $x$ and $y$.  Then
$$H_Z(N+1) - H_Z(N) = {N+2 \choose 1} = N+2$$
Thus, $\text{gin}(I_{Z_m})_{N+1}$ contains all monomials of degrees $N+1$ in the variables $x$ and $y$.

Now we need to determine $H_Z(N) - H_Z(N-1)$ and show that it is less than $N+1$ (that is, $\text{gin}(I_{Z_m})_N$ does not contain all monomials in $x$ and $y$ of degree $N$).  

Consider 
$$F_{N-1} = (17^2m'-1)L - 17 \cdot 6m'(E_1+\cdots +E_8).$$
Recall that one cycle is equal to $[Y]=48e_0 - 17(e_1+\cdots +e_7).$  It is easy to see that exactly 6 cycles can be subtracted off of $[F_{N-1}]$ to obtain $[H_{N-1}]$.  We have
$$[H_{N-1}] = (17^2m'-1-6\cdot 48)e_0 - (17\cdot 6m'-17\cdot 6)(e_1+ \cdots + e_8)$$
so, by Corollary \ref{cor:h0ofFtNEFF},
$$h^0(X, \mathcal{F}_t) = \frac{289}{2}m'^2-\frac{527}{2}m'+120$$

From Step 1 we know that ${F}_N$ is numerically effective so
$$h^0(X, \mathcal{F}_N) =  \frac{289}{2}m'^2+\frac{51}{2}m'+1$$

Thus, 
$$H_Z(N)-H_Z(N-1) = 238m'-119<N+1 = 289m'+1$$
 and not all monomials in $x$ and $y$ of degree $N$ are contained in $\text{gin}(I^{(m)})$.  Therefore, the largest degree generator of $\text{gin}(I_{Z_m})$ is of degree $N+1 = \frac{17}{6}m+1$ when $m$ is divisible by $17\cdot 6$.

\textbf{Step 5:}  By Step 4, the highest degree generator of $\text{gin}(I^{(m)})$ is of degree $\frac{17}{6}m+1$ when $m$ is divisible by $17\cdot 6$.  By Step 3, the smallest degree element of $\text{gin}(I^{(m)})$ is of degree $\alpha(m) = \frac{48m}{17}$ when $m$ is divisible by $17 \cdot 6$.  Thus, the intercepts of the limiting shape of the symbolic generic initial system of $I$ are $(0, \frac{17}{6})$ and $(\frac{48}{17},0)$.  Since
$$\frac{17}{6} \cdot \frac{48}{17} = 8,$$
Corollary \ref{cor:InterceptsDetermineShape} tells us that the limiting polytope is defined by the line through the intercepts and is as claimed in Theorem \ref{thm:polytopefor678}.

\section{The Symbolic Generic Initial System of 5 or Fewer Uniform Fat Points in General Position}
\label{sec:smallnumber}

In this section we prove part (c) of the main theorem.

\newtheorem*{thm:mainthmc}{Theorem \ref{thm:mainthm} (c)}
\begin{thm:mainthmc}
\label{thm:ShapeofFiveorLess}
Suppose that $1<r \leq 5$ and $I$ is the ideal of $r$ points in general position.  Then the limiting polytope of the symbolic generic initial system $\{\text{gin}(I^{(m)})\}$ has a boundary defined by the line through the points $(2,0)$ and $(0, \frac{r}{2})$ when $\frac{r}{2} \geq 2$ and $(\frac{r}{2},0)$ and $(0, 2)$ when $\frac{r}{2}<2$.
\end{thm:mainthmc}

This is an immediate consequence of the following result of \cite{Mayes12d} since five or fewer points of $\mathbb{P}^2$ in general position lie on an irreducible conic.

\begin{thm}
Suppose that $I$ is the ideal of $r$ points in $\mathbb{P}^2$ lying on an irreducible conic.  Then the limiting polytope of the symbolic generic initial system $\text{gin}(I^{(m)})$ has a boundary defined by the line through the points $(2,0)$ and $(0, \frac{r}{2})$ when $\frac{r}{2} \geq 2$ and $(\frac{r}{2},0)$ and $(0, 2)$ when $\frac{r}{2}<2$.
\end{thm}

\section{Final Example}

The results presented here and in \cite{Mayes12a} may lead the reader to believe that the limiting polytope of any symbolic generic initial system is defined by a single hyperplane.  The following example shows that this does not hold even for ideals of points in $\mathbb{P}^2$.

\begin{example}
\label{example:DifferentShape}
Suppose that $I$ is the ideal of the $l+1 \geq 4$ points $p_1, \dots, p_l, p_{l+1}$ of $\mathbb{P}^2$ where $p_1, \dots p_l$ lie on a line and $p_{l+1}$ lies off of the line.  

\begin{prop}
\label{prop:lpointson}
Let $I$ be the ideal of $l+1$ distinct points of $\mathbb{P}^2$ where $l$ of the points lie on a line and suppose that $l(l-1)$ divides $m$.  Then the highest degree generator of $\text{gin}(I^{(m)})$ is of degree $lm$ and the lowest degree generator of $\text{gin}(I^{(m)})$ is of degree $2m-\frac{m}{l}$.
\end{prop}

\begin{proof}[{Idea of Proof}]
The proof of this proposition is similar to the work contained in Section \ref{sec:678} with the following considerations.  In this case, the blow-up $\pi: X \rightarrow \mathbb{P}^2$ of $p_1, \dots, p_{l+1}$ has exceptional curves with classes $[L-E_1-E_2 - \cdots -E_l]$ and $[L-E_i-E_{l+1}]$ for $i=1, \dots, l$ where $E_j = \pi^{-1}(p_j)$ and $L$ is the total transform of a general line in $\mathbb{P}^2$ (note that the exceptional curves are the total transforms of lines through the points $\mathbb{P}^2$; see \cite{Harbourne98}).  
\end{proof}

If $P$ is the limiting shape of the symbolic generic initial system $\{ \text{gin}(I^{(m)}) \}_m$, then Proposition \ref{prop:lpointson} implies that the boundary of $P$ has $y$-intercept 
$$\lim_{ m \rightarrow \infty} \frac{lm}{m} = l$$
and $x$-intercept
$$\lim_{m \rightarrow \infty} \frac{2m- \frac{m}{l}}{m} = 2- \frac{1}{l}.$$

If the boundary of $P$ was defined by the line through these intercepts, the volume under of $P$ would be 
$$\text{vol}(Q) = \frac{(l)(2-\frac{1}{l})}{2} = l -\frac{1}{2}.$$
However, by Lemma \ref{lem:volume}, the volume under of $P$ must be $\frac{l+1}{2}$ which is strictly smaller than $l-\frac{1}{2}$ ($l \geq 3$).  Thus, $P$ is not defined by the line through the intercepts.  In fact, one can prove that the limiting polytope $P$ is the one shown in Figure \ref{fig:lpointsonpolytope}.

\end{example}

\vspace{0.2in}

\begin{figure}
\begin{center}
\includegraphics[width=4.5cm]{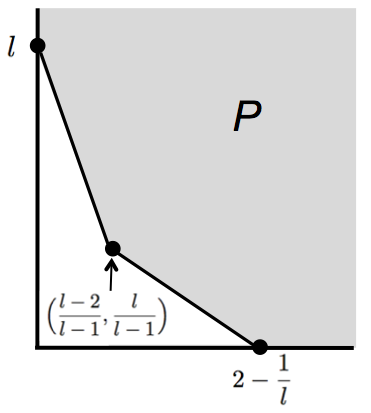}
\end{center}
\caption{The limiting polytope $P$ of the symbolic generic initial system of the ideal of $l$ points on a line and one point off.}
\label{fig:lpointsonpolytope}
\end{figure}

\bibliography{GeneralPositionBib}
\bibliographystyle{amsalpha}
\nocite{*}

\end{document}